\documentclass[12pt]{amsart}
\usepackage{amssymb,amsmath,geometry,latexsym,
graphics,tabularx,shapepar,enumerate,mathtools}

\newcommand{\wpi}{\widetilde{\pi}}
\newcommand{\GL}{\operatorname{GL}}

\newcommand{\ZZ}{\mathbb{Z}}
\newcommand{\FF}{\mathbb{F}}
\newcommand{\CC}{\mathbb{C}}
\newcommand{\QQ}{\mathbb{Q}}
\newcommand{\RR}{\mathbb{R}}
\newcommand{\GG}{\mathbb{G}}
\newcommand{\KK}{\mathbb{K}}

\newcommand{\NN}{\mathbb{N}}

\newcommand{\TT}{\mathbb{T}}

\usepackage[OT2,T1]{fontenc}
\DeclareSymbolFont{cyrletters}{OT2}{wncyr}{m}{n}
\DeclareMathSymbol{\Sha}{\mathalpha}{cyrletters}{"58}

\newtheorem{Theorem}{Theorem}
\newtheorem{Lemma}[Theorem]{Lemma}
\newtheorem{Proposition}[Theorem]{Proposition}
\newtheorem{Question}[Theorem]{Question}

\newtheorem{Definition}[Theorem]{Definition}
\newtheorem{Remark}[Theorem]{Remark}
\newtheorem{Conjecture}[Theorem]{Conjecture}

\date{March, 2017}

\title[On Schanuel conjecture]{On a variant of Schanuel conjecture for the Carlitz exponential}

\author{F. Pellarin}
\address{Federico Pellarin: Institut Camille Jordan, UMR 5208 Site de Saint-Etienne, 23 rue du Dr. P. Michelon, 42023 Saint-Etienne,
France}
\email{federico.pellarin@univ-st-etienne.fr}


\thanks{This project has received funding from the European Research Council (ERC)
under the European Union's Horizon 2020 research and innovation programme under
the Grant Agreement No 648132.}


\begin{document}

\maketitle

\renewcommand{\abstractname}{Résumé}

\begin{abstract}

Nous introduisons et décrivons une variante de la conjecture de Schanuel dans le cadre de l'exponentielle de Carlitz sur des algèbres de Tate et de fonctions similaires. Un autre objectif de ce travail est de stimuler des possibles investigations en transcendance et indépendance algébrique en caractéristique non nulle.
\end{abstract}

\renewcommand{\abstractname}{Abstract}

\begin{abstract}

We introduce and discuss a variant of Schanuel conjecture in the framework of 
the Carlitz exponential function over Tate algebras and allied functions.
Another purpose of the present paper is to widen the horizons of possible 
investigations in transcendence and algebraic independence in positive characteristic.

\end{abstract}

\section{Introduction}

Schanuel's conjecture is an unproven statement predicting the behavior of 
the intersections with algebraic sub-varieties defined over $\QQ$ of a certain analytic 
subvariety of $\GG_a^n(\CC)\times\GG_m^n(\CC)$ of dimension $n$,
built on the graph of the classical exponential function. Somehow, the conjecture
expects that these intersections behave in the simplest possible way: 

\begin{Conjecture}[Schanuel]\label{originalconjectureofschanuel}
Let $u_1,\ldots,u_n$
be complex numbers which are linearly independent over $\QQ$. Then 
the transcendence degree over $\QQ$ of the subfield of $\CC$: 
$$\QQ(u_1,\ldots,u_n,e^{u_1},\ldots,e^{u_n})$$
is $\geq n$.
\end{Conjecture}
 This conjecture first appeared in print in Lang's book \cite{LAN}.
It is surprising to see such a syntactically simple statement governing an intricate constellation
of results of independence of classical mathematical constants. 
We mention that
the {\em Lindemann-Weierstrass Theorem} confirms Schanuel's conjecture in the 
case of $u_1,\ldots,u_n$ algebraic numbers, and can be seen as a particular case of
Siegel-Shidlowski Theorem for the values of Siegel $E$-functions
at algebraic numbers (see Beukers' \cite{BEU}). Baker's Theorem on linear forms of logarithms of algebraic
numbers, on the other hand, partially clarifies the case in which $e^{u_1},\ldots,e^{u_n}$
are algebraic numbers. Outside these two cases, $u_1,\ldots,u_n$ algebraic or
$e^{u_1},\ldots,e^{u_n}$ algebraic, fragmentary information is available. Among those, a corollary of Nesterenko's Theorem asserts that the numbers $\pi$ and $e^\pi$ are algebraically
independent over $\QQ$. We do not give precise references for these statements.
Instead, we refer to the survey of Waldschmidt \cite{WAL} and its detailed 
bibliography. The paper of Scanlon \cite{SCA} gives a nice introduction of related topics.

As a final, maybe less known and more recent consequence of Schanuel's conjecture, we also mention that
in \cite[Theorem 1.6]{MAR}, Marker deduces from the Schanuel conjecture
that every non-zero exponential polynomial $P(x,e^x)\in\QQ[x,e^x]$ such that $P(X,Y)$ depends on both $X$ and $Y$, has infinitely many algebraically independent
roots in $\CC$. In particular, one deduces, from Schanuel's conjecture, that any finite set of
distinct fixed points of the exponential are algebraically independent over $\QQ$. 
\medskip

In the present paper, we shall discuss variants of another conjecture, very similar to Schanuel's conjecture in aspect, but in the framework of function fields of positive characteristic. 
Although at first sight, our interest in this framework may seem rather artificial (why looking 
at an analogue conjecture while we already have so many  interesting open problems in the framework of the 
classical Schanuel conjecture?), we hope that
the reader, after having looked into the present paper, will be finally convinced that some new structures emerge in this framework which would have remained hidden in the classical setting,
more related to the theory of difference algebra and fields, perhaps giving a new view on the original Schanuel conjecture itself.

This text grew up from the tentative of the author to understand, in the viewpoint of 
the theory of transcendence and algebraic independence,  the meaning of the fact that certain $L$-values 
introduced in \cite{PEL} also behave as "Stark-Anderson units" and are thus sent to 
a polynomial by the Carlitz exponential function, extended over Tate algebras (see \cite{ANG&PEL1,APT,APT2}). Our main purpose is to propose a statement
which could play the role of a Schanuel conjecture in this setting; on the other hand,
we warn the reader that the present paper does not contain substantial mathematical proofs in this direction.

It is a pleasure to dedicate the present work to David Goss, that helped in a fundamental way to shape the intuition in these investigations and constantly encouraged the progression of these works with determinant energy, and to Michel Waldschmidt, that inspired the ``transcendental thinking" of the author, encouraged and accompanied him constantly and generously, from his first steps, as a number theorist. 

\subsection{Some background}
We set $A=\FF_q[\theta]$
and $K=\FF_q(\theta)$, where $\FF_q$ is the finite field with $q$ elements and $\theta$ is an indeterminate.
In all the following, we also denote by $p$ the characteristic of $\FF_q$ so that $q=p^e$ for an 
integer $e>0$.
We denote by $K_\infty$ the local field completion of $K$ at the infinity place of $K$; then
$K_\infty$ can be identified with the field of formal Laurent series $\FF_q((1/\theta))$, 
with valuation $v_\infty$ normalized by $v_\infty(\theta)=-1$.
Let $K_\infty^{ac}$ be an algebraic closure of $K_\infty$. Then we denote by 
$\CC_\infty$ the completion of $K_\infty^{ac}$ for the unique 
extension of $v_\infty$. Note that $\CC_\infty$ is a $K_\infty^{ac}$-vector space
of infinite dimension, and that $K_\infty^{ac}$ is of infinite dimension over $K_\infty$.
The field $\CC_\infty$ carries a unique extension 
$$\CC_\infty\xrightarrow{v_\infty}\QQ\cup\{\infty\}$$
of the valuation $v_\infty$. Sometimes, we will also use the associated norm 
$|\cdot|=q^{-v_\infty(\cdot)}$.
Just as the ring $\ZZ$ is discrete and co-compact in $\RR$,
we have that the $\FF_q$-algebra $A$ is discrete and co-compact in $K_\infty$; note also that the infinity place
is the only place of $K$ with this property.

\subsubsection{The Carlitz module}\label{thecarlitzmodule}
Let $\iota:A\rightarrow B$ be a commutative $A$-algebra. Then $B$ is equipped with the $\FF_q$-algebra endomorphism
$\tau:B\rightarrow B$ which sends $b\in B$ to $b^q\in B$ (the raising to the power $q$ is relative 
to the algebra structure of $B$ and is $\FF_q$-linear). Let $B\{\tau\}$ be the skew ring of 
finite sums $\sum_{i\geq 0}b_i\tau^i$ with $b_i\in B$ for all $i$, and product
defined by the rule $\tau b=\tau(b)\tau$ for $b\in B$; the ring $B\{\tau\}$
acts on $B$, as well as on any $B$-algebra, by {\em evaluation}. If $f=\sum_i\tau^i\in B\{\tau\}$, then for all $b\in B$,
we define the evaluation of $f$ at $b$ as: $$f(b):=\sum_if_i\tau^i(b)=\sum_if_ib^{q^i}.$$ 
The 
{\em Carlitz module $C(B)$ over $B$} is by definition the 
$A$-module whose underlying $\FF_q$-vector space is $B$, in which the multiplication 
by $\theta\in A$, sufficient to define the entire $A$-module structure, is given 
by the evaluation of the skew polynomial $\iota(\theta)+\tau$. We shall write
$C_a(b)$ for the multiplication of an element $b$ of $B$ by $a\in A$. For example, we have 
$$C_\theta(b)=\iota(\theta)b+\tau(b),\quad b\in B.$$
In particular, $\CC_\infty$ is an $A$-algebra with $\iota$ the identity map, and the above construction
gives rise to the $A$-module $C(\CC_\infty)$.

\subsubsection{The Carlitz exponential}\label{thecarlitzexponential}
The Carlitz exponential $$\exp_C:\CC_\infty\rightarrow\CC_\infty$$ allows to {\em analytically uniformize} $C(\CC_\infty)$. In what concerns the foundation theory of this function,
we are going to follow Goss' treatise \cite{GOS} which can be consulted by any interested reader.
We first remark, to define this function, that there is, available in the theory over $A$,
an analogue of the sequence of numbers $n!$, defined as follows:
$$d_n=\prod_a a,$$
where the product runs over the monic polynomials $a$ of $A$ of degree $n$
(this is, more properly, an analogue of the number $q^n!$). We set, for $z\in\CC_\infty$,
$$\exp_C(z):=\sum_{n\geq 0}d_n^{-1}\tau^n(z)\in\CC_\infty.$$
One verifies that $v_\infty(d_n)=nq^n$, from which we deduce that
$\exp_C$ is an entire (hence necessarily surjective by simple considerations of Newton polygons), $\FF_q$-linear map $\CC_\infty\rightarrow\CC_\infty$.
Since it is entire, its kernel $\Lambda$, an $\FF_q$-vector space, determines $\exp_C$
uniquely, and we have the convergent Weierstrass product expansion 
\begin{equation}\label{weierstrassproduct}
\exp_C(z)=z\prod_{\lambda\in\Lambda\setminus\{0\}}\left(1-\frac{z}{\lambda}\right),\quad z\in\CC_\infty.\end{equation}

This function $\exp_C$ in fact is the unique entire $\FF_q$-linear map $\CC_\infty\xrightarrow{F}\CC_\infty$ which satisfies
$F'=1$, and which induces an exact sequence of $A$-modules:
$$0\rightarrow\Lambda\rightarrow\CC_\infty\rightarrow C(\CC_\infty)\rightarrow0,$$
where $\Lambda\subset\CC_\infty$ is an $A$-sub-module of $\CC_\infty$.
The study of the Newton polygon of $\exp_C$ implies that $\Lambda$ is of rank one.
In particular, there exists an element $\wpi\in\CC_\infty^\times$ such that
$\Lambda=\widetilde{\pi}A$; it is defined up to multiplication by an element of $\FF_q^\times=\FF_q\setminus\{0\}$.
Note that for all $a\in A$ and $z\in\CC_\infty$,
$$C_a(\exp_C(z))=\exp_C(az),$$ hence providing an analytic isomorphism of $A$-modules
$$C(\CC_\infty)\cong\frac{\CC_\infty}{\widetilde{\pi}A}.$$
The element $\widetilde{\pi}\in\CC_\infty$ can we constructed explicitly by limit processes
in several ways. For instance, we recall from \cite{GOS} that $\widetilde{\pi}$ is the value in $\CC_\infty$ of a convergent infinite product
\begin{equation}\label{pi}
\widetilde{\pi}:=-(-\theta)^{\frac{q}{q-1}}\prod_{i=1}^\infty(1-\theta^{1-q^i})^{-1}\in (-\theta)^{\frac{1}{q-1}}K_\infty,
\end{equation}
uniquely defined up to the choice of a root $(-\theta)^{\frac{1}{q-1}}$. It has been proved in a variety of
ways (see \cite{PEL0} for a survey) that $\widetilde{\pi}$ is moreover transcendental over
$K$; the first proof of which was obtained by Wade in \cite{WAD}.

\section{A Carlitzian analogue of Schanuel's conjecture}

The following conjecture is due to Laurent Denis (see \cite{DEN}).

\begin{Conjecture}[Denis]\label{CarlitzianSchanuelConj}
Let $u_1,\ldots,u_n$ be elements of $\CC_\infty$ which are $A$-linearly independent.
Then the transcendence degree over $K$ of the sub-field of $\CC_\infty$ 
$$K(u_1,\ldots,u_n,\exp_C(u_1),\ldots,\exp_C(u_n))$$
is $\geq n$.
\end{Conjecture}
The case in which $u_1,\ldots,u_n$ are algebraic over $K$ has been solved by Thiery in \cite{THI}
and can be viewed as the analogue of Lindemann-Weierstrass theorem for the Carlitz exponential.
Note that the analogue for the Carlitz exponential of the Theorem of Hermite-Lindemann already 
appears in Wade \cite{WAD} (see also \cite[Theorem 2.5]{PEL0}). In the arithmetic theory of 
function fields in positive characteristic, what it is usually called the "analogue of Hermite-Lindemann Theorem" is a corollary of a very general transcendence result by Yu in \cite{YU}
(see also \cite[Theorem 2.2]{PEL0}).
Years later, Papanikolas solved, in \cite{PAP}, the case in which $\exp_C(u_1),\ldots,\exp_C(u_n)$ are algebraic over $K$ by using the "criterion of linear independence of Anderson-Brownawell-Papanikolas". Denis also obtained \cite{DEN1}, at the same time as Papanikolas, 
a partial result in this direction, by using a characteristic $p$ variant of "Mahler's method".
Additionally, in \cite[Corollaire 2 (a)]{DEN}, Denis proved that $\widetilde{\pi}$ and $\widetilde{e}:=\exp_C(1)$
are algebraically independent over $K$ if $q\geq3$ (see \cite{PEL} for an overview of these results).

\subsection{Digression: strengthening}
We mention, briefly, a natural way to reinforce Conjecture \ref{CarlitzianSchanuelConj},
but this has to be considered as a digression, since the nature of the statements we are 
primarily interested in, is different.
Denis, in \cite{DEN}, proposes a strengthening of Conjecture \ref{CarlitzianSchanuelConj}
in order to give an interpretation
of a multitude of results of algebraic independence he was obtaining, 
not only for special constants related to the Carlitz exponential, but also for their 
derivatives in the variable $\theta$. Although not central for our further investigations,
these strenghtenings tell us that in the present settings (over $\CC_\infty$), there are more apparent 
structures that there seems to be over $\CC$, and this allows us to 
propose more elaborate statements.
 
Denote by $\exp_C^{(i)}\in K[[z]]$ the $i$-th higher derivative of 
$\exp_C\in K[[z]]$ with respect to $\theta$ (note indeed that the formal series $\exp_C$
can also be viewed as a double formal series in $\FF_q[[\theta^{-1},z]]$); it is an $\FF_q$-linear entire function.
We recall that $p$ is the prime dividing $q$, that is, the characteristic of $\FF_q$.
\begin{Conjecture}[Denis]\label{DenisSchanuelConj}
Let $u_1,\ldots,u_n\in\CC_\infty$ be elements which are $A$-linearly independent.
Then the transcendence degree over $K$ of the sub-field of $\CC_\infty$ 
$$K(u_1,\ldots,u_n,\exp_C(u_1),\ldots,\exp_C(u_n),\ldots,\exp_C^{(p-1)}(u_1),\ldots,\exp_C^{(p-1)}(u_n))$$
is $\geq pn$.
\end{Conjecture}
In fact, the original statement in \cite{DEN} was for 
$u_1,\ldots,u_n\in K_\infty^{ac}$. However, we do not see any reason for which 
Conjecture \ref{DenisSchanuelConj} should fail for $u_i\in\CC_\infty\setminus K_\infty^{ac}$ for some $i$.
This conjecture obviously implies Conjecture \ref{CarlitzianSchanuelConj}.

\begin{Remark}{\em Note that $\widetilde{\pi}$ belongs to the separable closure $K_\infty^{\text{sep}}$ of $K_\infty$ in $\CC_\infty$. It is easy to see that the usual $\FF_q$-linear higher derivatives in $\theta$, defined by $D_n(\theta^m)=\binom{m}{n}\theta^{m-n}$, extend in an unique way to $\FF_q^{ac}$-linear higher derivatives 
$K_\infty^{\text{sep}}\rightarrow K_\infty^{\text{sep}}$ (with $\FF_q^{ac}$ denoting the algebraic closure of 
$\FF_q$ in $\CC_\infty$). P. Voutier informed the author \cite{VOU} of a work of D. Brownawell and A. van der Poorten, in which they proved that $\widetilde{\pi},D_1(\widetilde{\pi}),\ldots$ are algebraically independent over $K$, hence suggesting yet another way in which one could generalize Conjecture \ref{CarlitzianSchanuelConj}.}

\end{Remark}

\section{The basic settings for our operator-theoretic conjecture}

In the following, we will keep focusing on statements similar to 
Conjecture \ref{CarlitzianSchanuelConj}.
We want to formulate - this is the aim of the present paper - an {\em operator-theoretic conjecture} (see Conjecture \ref{conjecturefunctional} below) implying Conjecture \ref{CarlitzianSchanuelConj}.
 In order to do so, we 
first need to introduce certain difference Banach algebras.  In \S \ref{extensionexp}, we will extend 
the Carlitz exponential to these algebras. This will give rise to 
several notions of independence, generalizing algebraic independence over $K$ essential for our statement, and 
studied in \S \ref{dimensions}.
The statement of our conjecture will appear in \S \ref{refinedschanuel} and we will give
examples of application in \S \ref{someexamples}.

We shall start with the algebra $\CC_\infty[t_1,\ldots,t_s]$ for variables $t_1,\ldots,t_s$.
If $q=p^e$ for an integer $e>0$, then we set $\tau=\mu^e$ which is $\FF_q[t_1,\ldots,t_s]$-linear.
It is customary at this point, to take, when it is well defined, the completion of the above $\CC_\infty$-algebra
for the unique extension of the valuation $v_\infty$ which is trivial over $\FF_q[t_1,\ldots,t_s]$,
giving rise to a {\em Tate algebra}.
Observe that if $s=0$, then we just have the field $\CC_\infty$.
Let $$\CC_\infty[t_1,\ldots,t_s]\xrightarrow{v_\infty}\QQ\cup\{\infty\}$$ be the unique extension of the valuation $v_\infty$ over $\CC_\infty$ which is trivial
over $\FF_q[t_1,\ldots,t_s]$. The completion of $\CC_\infty[t_1,\ldots,t_s]$ with respect to this valuation
is the standard {\em $s$-dimensional Tate algebra} denoted by $\TT_s$ in all the following.
It is also called, in several papers, the {\em free $s$-dimensional affinoid algebra over $\CC_\infty$}.
It is well known that $\TT_s$ is a ring which is Noetherian, factorial, of Krull dimension $s$ (see \cite{BGR} for the general theory of these algebras).
It is isomorphic to the $\CC_\infty$-algebra of the formal series
$$f=\sum_{i_1,\ldots,i_s\geq 0}f_{i_1,\ldots,i_s}t_1^{i_1}\cdots t_s^{i_s}\in\CC_\infty[[t_1,\ldots,t_s]]$$
which satisfy $$\lim_{\min\{i_1,\ldots,i_s\}\rightarrow\infty}f_{i_1,\ldots,i_s}=0.$$ Thus, we have, for $f$ a formal series of $\TT_s$ expanded as above,
and non-zero, that
$$v_\infty(f)=\inf_{i_1,\ldots,i_s}v_\infty(f_{i_1,\ldots,i_s})=\min_{i_1,\ldots,i_s}v_\infty(f_{i_1,\ldots,i_s}).$$ We also set, for convenience, 
$$\|\cdot\|:=q^{-v_\infty(\cdot)}$$ and $0=\|0\|=q^{-\infty}$.
In all the following, we are going to view the algebras $\TT_s$ as one embedded in the other, so that
$$\CC_\infty=\TT_0\subset\TT_1\subset\cdots\subset\TT_s\subset\cdots.$$
It is easy to see that, for all $i>0$,
$$\TT_{i}=\left\{f=\sum_{j\geq 0}f_jt_i^j;f_j\in\TT_{i-1},v_\infty(f_j)\rightarrow\infty\right\},$$ where $v_\infty$ denotes here the Gauss valuation
extending the valuation of $\CC_\infty$ over $\TT_{i-1}$.

Then the map $\CC_\infty\xrightarrow{\mu}\CC_\infty$, $x\mapsto\mu(x)=x^p$ extends, uniquely, to a continuous, open $\FF_p[t_1,\ldots,t_s]$-linear 
automorphism of $\TT_s$ (for all $s$) such that, for all $f\in\TT_s$, we have that $$v_\infty(\mu(f))=pv_\infty(f).$$ It is also easy to prove that the subring $\TT_s^{\mu=1}=\{f\in\TT_s;\mu(f)=f\}$
is equal to $\FF_p[t_1,\ldots,t_s]$ (all these properties are proved in detail in \cite{APT}).

\subsection{Some complete difference fields}
\label{completedifferencefields} It is suitable to also have, at hand, complete difference fields containing $\TT_s$, not just difference algebras. 
Let $L$ be any commutative field. We denote by $L\langle\langle\QQ\rangle\rangle_\infty$ the set of formal series
$$\sum_{i\in\mathcal{I}}c_i\theta^{-i},\quad c_i\in L,$$
where $\mathcal{I}\subset\QQ$ is a {\em well-ordered} subset, that is,
any non-empty subset has a minimum element. Then with the natural valuation $v_\infty$ trivial over $L$
and such that $v_\infty(\theta)=-1$, it is well known that $L\langle\langle\QQ\rangle\rangle_\infty$ is a 
valued field which is complete and has residual field $L$ (it is in fact henselian). Moreover, it has no proper immediate extensions and, if $L$ is algebraically closed, then $L\langle\langle\QQ\rangle\rangle_\infty$
is algebraically closed. Since the residue field of $\CC_\infty$, $\FF_q^{ac}$, is algebraically 
closed, there is an isometric map of $\CC_\infty$ in 
$\FF_q^{ac}\langle\langle\QQ\rangle\rangle_\infty$, and therefore, if $L=\FF_q^{ac}(t_1,\ldots,t_s)$, then
$\CC_\infty(t_1,\ldots,t_s)$ maps isometrically in $$\mathbb{K}_s:=\FF_q^{ac}(t_1,\ldots,t_s)\langle\langle\QQ\rangle\rangle_\infty.$$
The complete field $\mathbb{K}_s$ has thus residue field $\FF_q^{ac}(t_1,\ldots,t_s)$
and valuation group $\QQ$.
In particular, there is an isometric map from the completion of the fraction field of $\TT_s$ to 
$\KK_s$.

The map $\mu:\KK_s\rightarrow \KK_s$ which sends an element $x=\sum_{i\in\mathcal{I}}c_i\theta^{-i}\in
\KK_s$ to the element $\mu(x)=\sum_{i\in\mathcal{I}}\mu(c_i)\theta^{-pi}$ extends the previously introduced $\FF_p[t_1,\ldots,t_s]$-linear automorphism of $\TT_s$ (seen as embedded in $\mathbb{K}_s$ as indicated above) is a continuous, open $\FF_q(t_1,\ldots,t_s)$-linear automorphism. Further, we have the identity
$$\mathbb{K}_s^{\mu=1}=\{f\in\mathbb{K}_s:\mu(f)=f\}=\FF_p(t_1,\ldots,t_s)$$
which follows directly from the definition of the field $\KK_s$.
\subsection{Extensions of the Carlitz exponential function}\label{extensionexp}

We briefly recall some basic facts about the Carlitz exponential over Tate algebras, from \cite{APT}. We can construct a continuous, open $\FF_q[t_1,\ldots,t_s]$-linear endomorphism
$$\TT_s\xrightarrow{\exp_C}\TT_s,$$ by setting
$$\exp_C(f)=\sum_{i\geq 0}d_i^{-1}\tau^i(f)=\sum_{i\geq 0}d_i^{-1}\mu^{ei}(f),\quad f\in\TT_s.$$
Note that the restriction of $\exp_C$ to the subring $\TT_0=\CC_\infty\cong\CC_\infty.1\subset\TT_s$
returns the Carlitz exponential function defined in \S \ref{thecarlitzexponential}.
Following the arguments of \S \ref{thecarlitzmodule}, we endow $\TT_s$
with the structure of an $A[t_1,\ldots,t_s]$-module $C(\TT_s)$ in the following way.
The underlying $\FF_q[t_1,\ldots,t_s]$-module is just that of $\TT_s$. Moreover, 
the multiplication by $\theta\in A$ is given by $C_\theta=\theta+\tau$ and all this produces
a structure of $A[t_1,\ldots,t_s]$-module in an unique way.
\subsubsection*{Example}
To give a concrete example of how this $A[t_1,\ldots,t_s]$-module structure works, let us suppose that $s=1$. In this case, we more simply write $t=t_1$
and $\TT=\TT_1$.
Let us consider $f=t_1-\theta=t-\theta$, which belongs to $A[t]\subset \TT$. Then $\tau(f)=t-\theta^q$ and $$C_\theta(f)=t(\theta+1)-(\theta^2+\theta^q).$$
The following result is proved in \cite{ANG&PEL2}.
\begin{Proposition}
The map $\exp_C$ induces an exact sequence of $A[t_1,\ldots,t_s]$-modules:
\begin{equation}\label{exactsequencetatemodule}
0\rightarrow\widetilde{\pi}A[t_1,\ldots,t_s]\rightarrow\TT_s\xrightarrow{\exp_C} C(\TT_s)\rightarrow0.\end{equation}
\end{Proposition}
\begin{Remark}
{\em For $s=0$, we have seen that $\exp_C$  defines an entire function $\CC_\infty\rightarrow\CC_\infty$. In particular, $\exp_C$, as an entire function, is uniquely defined by the divisor of its zeroes and has the Weierstrass product expansion
(\ref{weierstrassproduct}) with $\Lambda=\widetilde{\pi}A$. However, for $s>0$, the extension of
$\exp_C$ to the Tate algebra $\TT_s$ that we have defined above, 
is no longer entire. For example, it has no Weierstrass product expansion over $\TT_s$ in contrast with the case $s=0$.}
\end{Remark}

More generally, we have the following.

\begin{Proposition}\label{propogenerals}
The Carlitz exponential $\exp_C$ gives rise to an exact sequence of
$\FF_q(t_1,\ldots,t_s)[\theta]$-modules:
$$0\rightarrow\widetilde{\pi}\FF_q(t_1,\ldots,t_s)[\theta]\rightarrow\mathbb{K}_s\xrightarrow{\exp_C} C(\mathbb{K}_s)\rightarrow0.$$
\end{Proposition}
Observe that the $\FF_q(t_1,\ldots,t_s)[\theta]$-module $C(\mathbb{K}_s)$ is well defined; the multiplication of
$f\in\mathbb{K}_s$
by $\theta$ for this module structure is $C_\theta(f)=\theta f+\tau(f)=\theta f+\mu^e(f)$, and all the operators are extended $\FF_q(t_1,\ldots,t_s)$-linearly. Restricted over the image of $\mathbb{T}_s$
in $\KK_s$, this exact sequence gives back the exact sequence (\ref{exactsequencetatemodule}).

\begin{proof}[Proof of Proposition \ref{propogenerals}]
Since $\KK_s$ is complete, $\exp_C$ is well defined. It is easy to see, writing down explicit 
generalized series of $\KK_s$, that if $g\in\KK_s$, then there exists a solution $f\in \KK_s$
of the equation $C_\theta(f)=\theta f+\tau(f)=g$. This means that $C(\KK_s)$ is $\theta$-divisible
and we can construct a continuous section of $\exp_C$; this implies that $\exp_C$ is surjective. Now, it is
immediate that $\FF_q(t_1,\ldots,t_s)[\theta]\widetilde{\pi}\subset\operatorname{Ker}(\exp_C)$.
To show the opposite inclusion, observe that, if $\mu\in\KK_s^\times$ is such that $\exp_C(\mu)=0$,
then, there exists $N\geq 0$ such that $v_\infty(\frac{\mu}{\theta^N})>-\frac{q}{q-1}$ and $\lambda=\exp_C(\frac{\mu}{\theta^N})\neq0$
is in the kernel of $C_{\theta^N}$. It is easy to show that this kernel is equal 
to $$\FF_q(t_1,\ldots,t_s)\exp_C\left(\frac{\widetilde{\pi}}{\theta}\right)\oplus\cdots\FF_q(t_1,\ldots,t_s)\exp_C\left(\frac{\widetilde{\pi}}{\theta^N}\right)$$ (recall that $\KK_s^{\tau=1}=\FF_q(t_1,\ldots,t_s)$).
Since $\exp_C$ induces an $\FF_q(t_1,\ldots,t_s)$-linear isometry of 
the disk $\{f\in\KK_s:v_\infty(f)>-\frac{q}{q-1}\}$, we can conclude.
\end{proof}

\subsection{A first conjectural statement}

We set $$K_s=\FF_p(\theta,t_1,\ldots,t_s).$$
This field will play the role of the field $K_0=\FF_p(\theta)$ (note that $K=K_0$ if and only if $q=p$; we hope that the fact that the notations $K$ and $K_s$ are so similar will not confuse the reader).
In the settings of \S \ref{extensionexp}, it is natural to state the following conjecture, which generalizes 
Denis' conjecture \ref{CarlitzianSchanuelConj}:

\begin{Conjecture}\label{improveddenisconj}
Let $u_1,\ldots,u_n$ be elements of $\mathbb{K}_s$ which are $A[t_1,\ldots,t_s]$-linearly independent.
Then the transcendence degree over $K_s$ of the sub-field of $\mathbb{K}_s$ 
$$K_s(u_1,\ldots,u_n,\exp_C(u_1),\ldots,\exp_C(u_n))$$
is $\geq n$.
\end{Conjecture}
We now explain why the above statement does not look so interesting; we set $s=1$ for commodity
(we write thus $\TT=\TT_1,t=t_1$ etc.).
The 
{\em Anderson-Thakur function} can be defined as the
element $$\omega=\exp_C\left(\frac{\widetilde{\pi}}{\theta-t}\right)\in\TT^\times.$$
It is the generator of the $\FF_q[t]$-module $\operatorname{Ker}(C_{\theta-t})\cap\TT$,
free of rank one (see \cite{ANG&PEL2}). 
Since 
\begin{equation}\label{differenceomega}
\tau(\omega)=(t-\theta)\omega
\end{equation} (this is equivalent to saying that $\omega\in \operatorname{Ker}(C_{\theta-t})$), we also deduce the next Proposition in quite an elementary way (see \cite{ANG&PEL2} for more details). 

\begin{Proposition}\label{omegagamma} The following properties hold:
\begin{enumerate}
\item We have the product expansion $$\omega=(-\theta)^{\frac{1}{q-1}}\prod_{i\geq 0}\left(1-\frac{t}{\theta^{q^i}}\right)^{-1},$$ convergent in $\TT^\times$.
\item $\omega$, identified with a function of the variable $t\in\CC_\infty$ with $\|t\|\leq 1$, extends to a meromorphic function over $\CC_\infty$ and has,
as unique singularities, simple poles at the points $t=\theta,\theta^q,\theta^{q^2},\ldots$. The residues 
can be explicitly computed. In particular, we have $\operatorname{Res}_{t=\theta}(\omega)=-\widetilde{\pi}$.
\item The function $1/\omega$ extends to an entire function $\CC_\infty\rightarrow\CC_\infty$
with unique zeros located at the poles of $\omega$. 
\end{enumerate}
\end{Proposition}
Since $\omega$ has infinitely many poles, it is transcendental over $K(t)$ and it is easy to 
see that $\widetilde{\pi}$, $\omega$ are algebraically independent over $K(t)$, but the conjecture 
\ref{improveddenisconj} in the case $n=1$ and $s=1$ only implies that 
one among the two elements $\widetilde{\pi}$ and $\omega$ is transcendental.

In fact, $\omega$ is transcendental over $K(t)$ but is also "$\tau$-algebraic" in the following sense: it satisfies the difference 
equation (\ref{differenceomega}) which is a kind of algebraic relation involving $\tau$, even though 
$\tau$ is not an algebraic morphism. What is missing to Conjecture \ref{improveddenisconj}
is the sensitiveness to such difference equations, and this is what we want to explore now.

\subsection{A suitable notion of independence}\label{dimensions}

Our operator-theoretic
generalization of Denis-Schanuel's conjecture \ref{DenisSchanuelConj} and of Conjecture \ref{improveddenisconj} that will take place in the 
valued field $\KK_s$. Before presenting it, we need to describe the relevant relations. 

We recall Conjectures \ref{originalconjectureofschanuel}, \ref{DenisSchanuelConj} and 
\ref{improveddenisconj}. In the simplest portrait, we have a complete {\em environment field} ($\CC$ in the 
first above mentioned conjecture, $\CC_\infty$ in the second conjecture and $\mathbb{K}_s$ in the third) a {\em countable base subfield of coefficients} ($\QQ$ in the 
first conjecture, $K$ in the second conjecture and $K_s$ in the third). Moreover, we have notions of {\em linear relation}
(over $\ZZ$, $A$ and $A[t_1,\ldots,t_s]$) and of  {\em independence} (over the fraction fields of these rings). 
Our statement will take place in the environment field $\KK_s$, concerns
the extension of the Carlitz exponential discussed in the previous section, and will vaguely sound as follows, with a (momentarily) non-specified countable subfield of coefficients $L_s\subset \KK_s$
and a non-specified notion of dependence over $L_s$:

\medskip

\noindent{\bf Conjecture} (Prototype)
{\em Let $f_1,\ldots,f_n$ be elements of $\KK_s$ which are $\FF_q(t_1,\ldots,t_s)[\theta]$-linearly independent.
Then, among the $2n$ elements $f_1,\ldots,f_n$ and $\exp_C(f_1)$, $\ldots,$ $\exp_C(f_n)$,
$n$ are ``$\mu$-independent" over $L_s$.}

\medskip

Of course, since our conjecture is taylored to imply Conjectures \ref{DenisSchanuelConj} and 
\ref{improveddenisconj},
we want that our relations over $L_s$ to extend algebraic dependence over $K=\FF_q(\theta)$, which is equivalent to
algebraic dependence over $\FF_p(\theta)$.
Now, we are going to suppose that $s\geq 1$.  In our $\mu$-difference settings, the difference field
extension $(\KK_s,\mu)$ of the difference field $(K_s,\mu)$,
there are several noticeable classical notions of independence over $K_s$ but we will see soon that they must be refined for our purposes.

\subsubsection{Transformal independence}
The first notion we want to discuss is that of {\em transformal independence} over $K_s$ (read Levin, \cite[\S 2.2]{LEV}); as announced, it will turn out to be too coarse. 

Let $L$ be
a $\mu$-difference subfield of $\KK_s$ containing $K_s$. 

\begin{Definition}{\em 
Elements 
$f_1,\ldots,f_n\in\KK_s$ are {\em transformally $\mu$-independent} over $L$ if the elements 
$$f_1,\mu(f_1),\ldots, f_2,\mu(f_2),\ldots,\ldots,f_n,\mu(f_n),\ldots$$ are algebraically independent
over $L$; the notions of {\em transformal $\mu$-dependence, transformal $\mu$-algebraicity} and {\em transformal $\mu$-transcendence} can be defined accordingly (see \cite{LEV}). 
This also leads to a notion of {\em transformal $\mu$-independence degree}. Let $f_1,\ldots,f_n$ be elements of $\KK_s$ and let us denote by:
$$\mathcal{L}=L(f_1,\ldots,f_n)_\mu$$ the smallest $\mu$-difference subfield of $\KK_s$
containing $L$ and $f_1,\ldots,f_n$. Then the transformal $\mu$-independence degree of $\mathcal{L}$ over $L$
$$\operatorname{transf}\deg_L(\mathcal{L})$$ is by definition the minimal cardinality
of a transformally $\mu$-independent subset of $\mathcal{L}$ over $L$. If 
$f_1,\ldots,f_n$ are transformally $\mu$-independent over $L$, then 
$$\operatorname{transf}\deg_L(\mathcal{L})=n.$$}
\end{Definition}

The notions above are not tailored for the statement we want to produce.
To illustrate this, observe that, since $c^p=\mu(c)$ for all $c\in\CC_\infty$, every element of $\CC_\infty$
is transformally $\mu$-algebraic over $K$, independently of the fact that it is algebraic over $K$ or not.
Worse, for all $c\in\CC_\infty$ and for all $f\in K_s$, the element 
$g=cf\in\CC_\infty(t_1,\ldots,t_s)$ is transformally $\mu$-algebraic. Indeed, 
$\mu(f)g^p=f^p\mu(g)$. We deduce from \cite[Theorem 4.1.2]{LEV} that every element of $\CC_\infty(t_1,\ldots,t_s)$
is transformally $\mu$-algebraic, no matter if the coefficients of the rational fractions 
are algebraic or not.

Note also that in particular, the subfield generated by the elements 
$x\in\KK_s$ which are trasformally $\mu$-algebraic over $K_s$ is not countable.
The generalization of Conjecture \ref{CarlitzianSchanuelConj}
(see Conjecture \ref{conjecturefunctional}) that we have in mind is of arithmetic nature, and is similar to 
\cite[Hypothesis ($\Sigma$) p. 252]{AX}, which is equivalent to the Schanuel conjecture
itself (see ibid., Theorem 2). There, a countable base subfield is required
(the field $\QQ$). It seems to us that for statements of arithmetic nature such as Schanuel's Conjecture (\footnote{Observe that we
are not considering here, statements such as \cite[(SP)]{AX}.}), it is natural to have a
countable base field; this means that the notion of transformal $\mu$-independence
is not suitable for our framework. We now make an attempt to give a 
more refined notion of relation.

\subsubsection{Analytically critical and regular $\mu$-polynomials}
We choose a finite set of symbols $$\underline{X}=(X_1,\ldots,X_n)$$ and we introduce infinitely many indeterminates 
$$\mu^0(X_i),\mu^1(X_i),\mu^2(X_i),\ldots,\quad i=1,\ldots,n.$$ To simplify the notation and to make it more natural, 
we shall write $X_i=\mu^0(X_i)$ and $\mu(X_i)=\mu^1(X_i)$. Let $L$ be a $\mu$-difference
field.
We consider the ring in infinitely many indeterminates
$$L[\underline{X}]_\mu=L[X_i,\mu(X_i),\mu^2(X_i),\ldots:i=1,\ldots,n],$$ that we turn into a
$\mu$-difference ring by setting 
$$\mu(\mu^j(X_i))=\mu^{j+1}(X_i).$$ The elements of $L[\underline{X}]_\mu$ are called the {\em $\mu$-polynomials in the 
symbols $\underline{X}$ with coefficients in $L$}. Any element $P\in L[\underline{X}]_\mu$
can be written in an unique way as
\begin{equation}\label{mupoly}
P(\underline{X})=\sum_{\underline{i}\in\NN^{(k+1)n}}c_{\underline{i}}X_1^{i_{1,0}}\mu(X_1)^{i_{1,1}}\cdots\mu^k(X_1)^{i_{1,k}}\cdots X_n^{i_{n,0}}\mu(X_n)^{i_{n,1}}\cdots\mu^k(X_n)^{i_{n,k}}, 
c_{\underline{i}}\in L,
\end{equation}
and the smallest $k\in\NN$ for which the expansion (\ref{mupoly}) holds is called the {\em depth} of $P$. 

Let $P$ be in $L[\underline{X}]_\mu$ as in (\ref{mupoly})
and let $\underline{f}=(f_1,\ldots,f_n)$ be an element of $\mathbb{K}_s^n$. Then the {\em evaluation of $P$ at $\underline{f}$}
is the element of $\mathbb{K}_s$:
$$P(\underline{f}):=\sum_{\underline{i}\in\NN^{(k+1)n}}c_{\underline{i}}f_1^{i_{1,0}}\mu(f_1)^{i_{1,1}}\cdots\mu^k(f_1)^{i_{1,k}}\cdots f_n^{i_{n,0}}\mu(f_n)^{i_{n,1}}\cdots\mu^k(f_n)^{i_{n,k}}.$$ If $P(\underline{f})=0$, we say that $f$ is a {\em zero} or a {\em root} of $P$. We choose 
a $\mu$-difference subfield $L'$ of $\KK_s$ containing the coefficients of a $\mu$-polynomial $P$.
The {\em $\mu$-variety of zeroes of $P$ in $L'$} is the subset of 
$\underline{f}\in L'{}^n$ such that $P(\underline{f})=0$. It is 
denoted 
by $Z_{L'}(P)$ or more simply, by $Z(P)$ when $L'=\mathbb{K}_s$. We can also consider 
varieties of simultaneous zeroes of a collection of $\mu$-polynomials.
We observe that $Z(P)=Z(\mu(P))$ for any $\mu$-polynomial.
Let $P$ be a $\mu$-polynomial. We denote by $P^\mu$ the $\mu$-polynomial 
obtained by applying $\mu$ to all the coefficients of $P$. Then, it is easy to see that $Z(P^\mu)=\mu(Z(P))$.


We set $D:=D(0,1)=\{z\in \CC_\infty:v_\infty(z)\geq0\}$. Let $n\geq 1$ be an integer.
We set $\underline{X}=(X_1,\ldots,X_n)$ and $\underline{z}=(z_1,\ldots,z_n)$.
We denote by $\TT_{\underline{z}}$ the Tate algebra 
$\TT_{\underline{z}}=\widehat{\CC_\infty[z_1,\ldots, z_n]},$
completion of the polynomial ring $\CC_\infty[z_1,\ldots, z_n]$ for the Gauss norm.

\begin{Definition}
{\em Let $P$ be a non-zero $\mu$-polynomial of $\KK_s[\underline{X}]_\mu$. 
We say that it is {\em analytically critical} if there exists a map
$$D^n\xrightarrow{\underline{F}} Z(P)\subset\KK_s^n,$$
$\underline{z}\in D^n\mapsto\underline{F}(\underline{z})=(F_1(\underline{z}),\ldots,F_n(\underline{z}))$
which is a {\em rigid immersion} in the following sense. The coordinate functions $\underline{F}=(F_1,\ldots,F_n)$ are series $$F_i=F_i(z_1,\ldots,z_n)=\sum_{\underline{k}=(k_1,\ldots,k_n)\atop
 k_i\geq 0, \forall i=1,\ldots,n}F_{i,\underline{k}}z_1^{k_1}\cdots z_n^{k_n}\in\widehat{\KK_s\otimes_{\CC_\infty}\TT_n},$$
(such that, for all $i,\underline{k}$, $F_{i,\underline{k}}\in\KK_s$ and $\|F_{i,\underline{k}}\|\rightarrow0$
as $|\underline{k}|=\sum_jk_j\rightarrow\infty$), and moreover, 
the Jacobian
$$J_{\underline{z}}(\underline{F})=\begin{pmatrix}
\frac{\partial F_1}{\partial z_1} & \ldots & \frac{\partial F_n}{\partial z_1} \\ 
\vdots &  & \vdots \\ 
\frac{\partial F_1}{\partial z_n} & \cdots & \frac{\partial F_n}{\partial z_n}
\end{pmatrix}$$
has everywhere maximal rank over $D^n$. A $\mu$-polynomial which is not analytically critical 
is said {\em analytically regular}.}
\end{Definition}
Equivalently, $P$ is analytically critical if a point $\underline{F}$ with elements of 
$\widehat{\mathbb{K}_s\otimes_{\CC_\infty}\TT_{\underline{z}}}$ as entries, satisfying the above conditions, belongs to
$Z_{\widehat{\mathbb{K}_s\otimes_{\CC_\infty}\TT_{\underline{z}}}}(P)$. From this, we
deduce that the set of analytically critical $\mu$-polynomials of $\mathbb{K}_s[\underline{X}]_\mu$
is a non-zero prime ideal.
\subsubsection{Examples}
Any polynomial $P\in\KK_s[\underline{X}]_\mu$
which is multi-homogeneous for the unique multi-graduation which assigns to $\mu^i(X_j)$
the degree $p^i$, is analytically critical as soon as it has 
a root $\underline{f}\in(\KK_s\setminus\{0\})^n$. Indeed, if $\underline{f}=(f_1,\ldots,f_n)\in\KK_s^n$
is a root, then also $(c_1f_1,\ldots,c_nf_n)$ for all $c_1,\ldots,c_n\in\CC_\infty$
is a root, due to the fact that $c^p=\mu(c)$ if $c\in\CC_\infty$, so that
there is a rigid immersion map $D^n\rightarrow Z(P)$.
In particular, 
the $\mu$-polynomial in one variable $\mu(X)-X^p$ is analytically critical.
Note that there also are analytically critical $\mu$-polynomials which are not homogeneous for the 
above graduations.



\subsubsection{Examples of analytically regular $\mu$-polynomials}

\begin{Lemma}\label{noncritical}
Any non-zero polynomial $P\in\KK_s[\underline{X}]$ is analytically regular.
Moreover, a $\mu$-polynomial $P\in\KK_s[\underline{X}]_\mu$ is analytically critical if and only if $\mu(P)$
is analytically critical. Similarly, $P$ is anaytically critical if and only if $P^\mu$ is analytically critical.
\end{Lemma}
\begin{proof}
It is easy to see, by the Jacobian criterion, that the components of any rigid immersion
are algebraically independent over $\KK_s$. This proves the first property. 
Now, recall that $Z(\mu(P))=Z(P)$ for any $\mu$-polynomial $P$; this suffices to show that $P$
is analytically critical if and only if $\mu(P)$ is analytically critical.
To prove the last property, recall that $Z(P^\mu)=\mu(Z(P))$.
Assume first that $P$ is analytically critical. Then, there is a rigid immersion $D^n\xrightarrow{\underline{F}} Z(P)$. We write $\underline{F}=(F_1,\ldots,F_n)$ 
with $F_i=\sum_{\underline{k}}F_{i,\underline{k}}z_1^{k_1}\cdots z_n^{k_n}$ with $F_{i,\underline{k}}\in\mathbb{K}_s$, for all $i$. We observe that $\mu$ induces a bijection of $D^n$
and, if we set $$\widetilde{F}_i:=\sum_{\underline{k}}\mu(F_{i,\underline{k}})z_1^{k_1}\cdots z_n^{k_n},\quad i=1,\ldots,n,$$
then $\widetilde{\underline{F}}=(\widetilde{F}_1,\ldots,\widetilde{F}_n)$ defines a rigid immersion
$D^n\rightarrow Z(P^\mu)$.
Similarly, let us suppose that $P^\mu$ is analytically critical, and let $\underline{F}=(F_1,\ldots,F_n)$
be a rigid immersion $D^n\rightarrow Z(\mu(P))$. Then,
setting now $$\widetilde{F}_i:=\sum_{\underline{k}}\mu^{-1}(F_{i,\underline{k}})z_1^{k_1}\cdots z_n^{k_n},\ldots, i=1,\ldots,n$$ defines a rigid immersion $D^n\rightarrow Z(P)$ and $P$ is analytically critical.
\end{proof}
Hence, a $\mu$-polynomial of depth zero is analytically regular, but these are not the only examples.
The next definition provides a larger class of 
analytically regular $\mu$-polynomials.
\begin{Definition}
{\em Let $P\in \KK_s[\underline{X}]_\mu$ be as in (\ref{mupoly}), and non-zero. We say that it is {\em tame}
if $c_{\underline{i}}\neq0$ implies that, writing $\underline{i}=(i_{l,j})_{0\leq l\leq n \atop 0\leq j\leq k}$,
we have that $0\leq i_{l,j}<p$ for all $l,j$. }
\end{Definition} 
In other, more loosely words, a tame $\mu$-polynomial is one such that no exponent exceeds, or is equal to $p$.
\begin{Lemma}
Every tame $\mu$-polynomial is analytically regular.
\end{Lemma}
\begin{proof}
We suppose again that $\underline{X}=(X_1,\ldots,X_n)$.
We shall proceed by induction over
the depth $k$; the statement is clear for depth $k=0$ (see Lemma \ref{noncritical}).
Let $P\in\KK_s[\underline{X}]_\mu$ be a tame $\mu$-polynomial with $\mu$-expansion as in (\ref{mupoly}) and depth $k>0$. By Lemma \ref{noncritical}, we can suppose that $P$ depends of the indeterminates $X_1,\ldots,X_n$
(more precisely, of the indeterminates $\mu^0(X_1),\ldots,\mu^0(X_n)$).
Let us suppose by contradiction that it is analytically critical. Then there is 
a rigid immersion $$D^n\xrightarrow{\underline{F}=(F_1,\ldots,F_n)}Z(P).$$
Since the Tate algebra $\TT_{\underline{z}}$ in the variables $z_1,\ldots,z_n$ is stable for the 
partial derivatives $\frac{\partial}{\partial z_i}$ for all $i$, the map $\underline{z}\mapsto P(\underline{F}(\underline{z}))$ is differentiable in each variable
$z_i$ (in the polydisk $D^n$). 
We note that $$\frac{\partial}{\partial z_j}(\mu^l(X_k)^i(\underline{F}(\underline{z})))=0,\quad j=1,\ldots,n,\quad i,l>0$$ (here, $\mu$ acts as the $p$-power on $\TT_{\underline{z}}$).
By the usual chain rule and the invertibility of the Jacobian 
of $\underline{F}$ we deduce 
that
$$\frac{\partial P}{\partial X_1}(\underline{F}(\underline{z}))=\cdots=\frac{\partial P}{\partial X_n}(\underline{F}(\underline{z}))=0.$$  Observe that 
$0<\deg_{X_i}(P)<p$ for all $i$. Then we have proved that there exists $i\in\{1,\ldots,n\}$ such that $\frac{\partial P}{\partial X_i}$ is tame and analytically critical. Iterating this process, 
we can successively construct a sequence of analytically critical tame polynomials such that 
the sequence of the total degrees in the indeterminates $X_1,\ldots,X_n$ is strictly decreasing
and we ultimately find a tame analytically critical $\mu$-polynomial of depth 
$k$ which does not depend of the indeterminates $X_1,\ldots,X_n$. Then, by Lemma \ref{noncritical},
$P=\mu^r(Q)$ for some $r>0$ and $Q$ tame, analytically critical of depth $<k$
which is impossible.
\end{proof}
\begin{Remark}{\em 
Note that a polynomial $P\in\KK_s[\underline{X}]$, that is, a $\mu$-polynomial of depth zero,
needs not to be tame. Also, the product of two tame $\mu$-polynomials needs not to be tame; however, there exists 
a natural $\KK_s$-algebra structure over the $\KK_s$-vector space of tame $\mu$-polynomials in $\KK_s[\underline{X}]_\mu$
which makes it isomorphic to $\KK_s[\underline{X}]$. See the appendix \S \ref{appendix} with the details of the construction (it will not be used in this paper). Note also that the maps $\CC_\infty^n\rightarrow\KK_s$ that tame polynomials define,
are in natural $\KK_s$-linear isomorphism with the polynomial maps $\CC_\infty^n\rightarrow\KK_s$
induced by polynomials of $\KK_s[\underline{X}]$.}\end{Remark}
\begin{Definition}{\em 
Let $L$ be a $\mu$-difference subfield of $\KK_s$
containing $K_s$.
Let $f_1,\ldots,f_n$ be elements of $\KK_s$. We say that they are {\em regularly $\mu$-independent over $L$}
if for all $P\in L[\underline{X}]_\mu$ analytically regular (hence non-zero), $P(\underline{f})\neq0$, where 
$\underline{f}=(f_1,\ldots,f_n)$. We say that $f_1,\ldots,f_n$ are {\em tamely $\mu$-independent over } $L$
if $P(\underline{f})\neq 0$ for all tame $\mu$-polynomial $P\in L[\underline{X}]_\mu$.}
\end{Definition}
The definition also produces notions of $\mu$-dependence, tame $\mu$-dependence, $\mu$-transcen\-dence,
$\mu$-algebraicity, etc.

If $f_1,\ldots,f_m$ are tamely $\mu$-dependent, then they are regularly $\mu$-dependent.
Also, if $f_1,\ldots,f_m$ are regularly $\mu$-dependent, then they are transformally $\mu$-dependent 
in the sense of Levin \cite{LEV}. But the reverse implications are all false.
We collect a few more properties in the following proposition.
\begin{Proposition}\label{someproperties}
Let $f_1,\ldots,f_n$ be elements of $\KK_s$. We have:
\begin{enumerate}
\item If $f_1,\ldots,f_n$ are regularly $\mu$-independent  over $L$, then they are algebraically independent
over $L$.
\item If $f_1,\ldots,f_n$ are in $\CC_\infty$ and are algebraically independent over $K$, then they are
regularly $\mu$-independent  over $K_s$.
\end{enumerate}
\end{Proposition}
There exist algebraic elements $f\in\KK_s$ which are transcendental over $K_s$
but tamely $\mu$-algebraic over $K_s$; examples are provided by certain torsion points of the 
Carlitz module, see \S \ref{torsion}.
\begin{proof}[Proof of Proposition \ref{someproperties}]
(1). It follows from Lemma \ref{noncritical}.
\noindent (2). We suppose to have elements $f_1,\ldots,f_n\in\CC_\infty$ and an analytically regular
$\mu$-polynomial $P\in K_s[\underline{X}]_\mu$ with $P(\underline{f})=0$ and
$\underline{f}=(f_1,\ldots,f_n)$. This implies, in particular, that the map $\CC_\infty^n\rightarrow\KK_s$
induced by $P$ is not identically zero. This map is further equal, by the fact that
$\mu$ and $x\mapsto x^p$ agree on $\CC_\infty$, to a polynomial map
induced by a polynomial $Q\in K_s[\underline{X}]$; since the map is not identically zero
by hypothesis, $Q$ is non-zero. Since $K^s$ is dense in $D^s$, there exists an
element $\underline{x}=(x_1,\ldots,x_s)\in D^s\cap K^s$ such that 
the evaluation of the coefficients of $Q$, which are rational functions of $K_s$, is well defined and does not map to $\{0\}$.
This yields a non-trivial algebraic dependence relation of $f_1,\ldots,f_n$ over $K$.\end{proof}

\subsection{A refinement of Denis-Schanuel's conjecture}\label{refinedschanuel}

We set $q=p^e$ with $e>0$ and we fix an integer $s\geq 0$. 
We recall that $K_s=\FF_p(\theta,t_1,\ldots,t_s)\subset\KK_s$. 
We denote by $L_s$ the smallest $\mu$-difference subfield of $\KK_s$
containing $K_s$ and all the $f\in\KK_s$ which are regularly $\mu$-algebraic over $K_s$;
by construction, it is countable.
The exponential function $\exp_C:\mathbb{K}_s\rightarrow C(\mathbb{K}_s)$ is defined by 
$$\exp_C=\sum_{i\geq 0}d_i^{-1}\tau^i=\sum_{i\geq 0}d_i^{-1}\mu^{ei}$$
as in \S \ref{extensionexp}.  We can now state our conjecture. 

\begin{Conjecture}[Operator-theoretic generalization of Denis-Schanuel's Conjecture]
\label{conjecturefunctional}
Let $f_1,\ldots,f_n$ be elements of $\mathbb{K}_s$ and let us write $g_i=\exp_C(f_i)\in\mathbb{K}_s$ for all $i=1,\ldots,n$. If $f_1,\ldots,f_n$
are $\FF_q(t_1,\ldots,t_s)[\theta]$-linearly independent, then $n$ among 
the elements $$f_1,\ldots,f_n,g_1,\ldots,g_n$$ are regularly $\mu$-independent over $L_s$.
\end{Conjecture}
We see that in the case $s=0$
or if $f_1,\ldots,f_n$ all belong to $\CC_\infty$, then this conjecture 
reduces to Conjecture \ref{CarlitzianSchanuelConj} by (2) of Proposition
\ref{someproperties}. Also, for any $s$, Conjecture \ref{conjecturefunctional}
implies Conjecture \ref{CarlitzianSchanuelConj} and Conjecture \ref{improveddenisconj}. We do not know if, reciprocally,
Conjecture \ref{CarlitzianSchanuelConj} implies Conjecture \ref{conjecturefunctional}
in analogy with the hypothesis ($\Sigma$) of \cite{AX}, which is equivalent to
the Schanuel conjecture; this looks unlikely.

\section{Some examples}\label{someexamples}

We are going to give some examples of consequences of our Conjecture \ref{conjecturefunctional}.
We first review basic properties of certain special functions that are used in our examples. 

\subsection{Torsion of the Carlitz exponential.}\label{torsion} Our Carlitz exponential 
$\exp_C:\KK_s\rightarrow C(\KK_s)$ has quite a rich torsion structure that we review here
(see \cite{APT} for more properties). Let $a\in \FF_q(t_1,\ldots,t_s)[\theta]$ be of degree $d\geq 0$ in $\theta$.
Then for all $j=0,\ldots,\deg_\theta(a)-1$, $\frac{\widetilde{\pi}\theta^j}{a}\in\KK_s$.
In particular $\exp_C\left(\frac{\widetilde{\pi}\theta^j}{a}\right)\in\KK_s$ and one sees that
$$C_a\left(\exp_C\left(\frac{\widetilde{\pi}\theta^j}{a}\right)\right)=0,\quad j=0,\ldots,\deg_\theta(a)-1.$$
In fact, the elements
$\exp_C\left(\frac{\widetilde{\pi}\theta^j}{a}\right)\in\KK_s$ constitute an $\FF_q(t_1,\ldots,t_s)$-basis
of the vector space underlying the $\FF_q(t_1,\ldots,t_s)[\theta]$-submodule $\operatorname{Ker}(C_a)\subset C(\KK_s)$. 
Every such a torsion point satisfies a non-trivial linear $\tau$-difference
equation with coefficients in $\FF_q(t_1,\ldots,t_s)[\theta]$, hence it satisfies a 
non-trivial linear $\mu$-difference equation with coefficients in $K_s$.
If moreover $a$ is a polynomial of $\FF_q[t_1,\ldots,t_s,\theta]$ which is monic in $\theta$, then $a^{-1}\in\TT_s$
and we have $\exp_C\left(\frac{\widetilde{\pi}\theta^j}{a}\right)\in\TT_s$ for all $j$. We see that,
for $j=0,\ldots,d-1$, the above elements span the rank $d$ free $\FF_q[t_1,\ldots,t_s]$-module 
$\operatorname{Ker}(C_a)\cap\TT_s$.

The simplest case (with $s=1$, case in which we write $t$ for $t_1$) is given by the 
Anderson-Thakur function and comes with the choice of $a=\theta-t$. Note that
it is also tamely $\mu$-algebraic; indeed, from (\ref{differenceomega}), we deduce that the $\mu$-polynomial $\mu^e(X)-(t-\theta)X=C_{\theta-t}(X)\in K_s[X]_\mu$ vanishes at $\omega$ and is tame.

\subsection{Zeta values.}
The so-called {\em Carlitz zeta values} are defined 
as follows, for $n\geq 1$ an integer, where
$A^+$ denotes the multiplicative monoid of monic polynomials of $A$, and the product runs over the irreducible polynomials $P\in A^+$:
$$\zeta_A(n)=\sum_{a\in A^+}a^{-n}=\prod_P\left(1-\frac{1}{P^n}\right)^{-1}\in 1+\theta^{-1}\FF_q[[\theta^{-1}]].$$
Carlitz essentially proved in  \cite{CAR} that 
\begin{equation}\label{carlitzidentity}\exp_C(\zeta_A(1))=1.\end{equation}
Conjectures \ref{CarlitzianSchanuelConj} and \ref{conjecturefunctional} immediately imply that $\zeta_A(1)$ is transcendental
over $K$ and indeed, this follows from the transcendence theory of 
the Carlitz module in a variety of ways. First of all, this can be viewed as a
consequence of the analogue of Hermite-Lindemann Theorem that can be found in \cite{WAD}.
Also, this follows from the Theorem of Papanikolas \cite{PAP}, see also 
Denis \cite{DEN1} and the survey \cite{PEL0} by the author, not to mention other 
proofs, with more diophantine, or automatic flavor.
More generally, Chieh-Yu Chang and Jing Yu have proved in \cite{CHYU}
the following result.
\begin{Theorem}[Chieh-Yu Chang \& Jing Yu]\label{changyu}
The element $\widetilde{\pi}$ and all the Carlitz zeta-values $\zeta_A(n)$ for all $n\geq 1$ not divisible 
by $p$ and $q-1$, are algebraically independent over $K$.
\end{Theorem}
Note that the elements of the above theorem generate the $K$-algebra
$$K(\widetilde{\pi},\zeta_A(1),\zeta_A(2),\ldots).$$ Indeed, on one side,
we have the ``trivial sum-shuffle relations" 
\begin{equation}\label{trivialsumshuffle}\zeta_A(pn)=\zeta_A(n)^p=\mu(\zeta_A(n)),\end{equation} 
and on the other hand, we have the result of Carlitz \begin{equation}\label{Carlitz}\zeta_A(k(q-1))\in K^\times\widetilde{\pi}^{k(q-1)}\end{equation} for all $k\geq 1$.
In other words, The so-called Bernoulli-Carlitz relations $\zeta_A(n)\in K^\times$
for $n>0$ such that $q-1\mid n$ and the relations (\ref{trivialsumshuffle}) exhaust all the algebraic dependence relations
between $\widetilde{\pi}$ and $\zeta_A(n)$ for all $n\geq 1$.

\subsubsection{Zeta values in Tate algebras}
We recall from \cite{APT} the construction of {\em Carlitz zeta values in $\mathbb{T}_s$}.
They are defined as follows, for $n\geq 1$ and an integer  $s\geq 0$:
$$\zeta_A(n;s)=\sum_{a\in A^+}a^{-n}a(t_1)\cdots a(t_s)=\prod_P\left(1-\frac{P(t_1)\cdots P(t_s)}{P^n}\right)^{-1}.$$ The above product converges in the complete topological subgroup 
$$1+\theta^{-1}\FF_q(t_1,\ldots,t_s)[[\theta^{-1}]]\subset\TT_s^\times.$$
The classical Carlitz zeta values are a special case of our construction
with $s=0$; in our notations, $\zeta_A(n)=\zeta_A(n;0)$. 
 In terms of the variables $t_1,\ldots,t_s$, these series define 
entire functions $\CC_\infty^s\rightarrow\CC_\infty$ (see \cite{ANG&PEL1}). Therefore, evaluation at $t_i=\theta^{q^{k_i}}$,
$i=1,\ldots,s$ and $k_i\in\ZZ$ makes sense and, for $n>0$, we have the identities
$$\zeta_A(n)=\zeta_A(n;0)=\zeta_A(n+q^{k_1}+\cdots+q^{k_s};s)|_{t_i=\theta^{q^{k_i}}}.$$
In this respect, we can view these functions as entire {\em interpolations of Carlitz zeta values}.
Moreover, note that, for all $m>0$, there exists $N\geq 0$ and $s$ such that 
$$\zeta_A(m)=\tau^N(\zeta_A(1;s))_{t_i=\theta,\forall i=1,\ldots,s}.$$ 
\subsubsection{Consequences of Conjecture \ref{conjecturefunctional}}

We begin by reviewing an important property
of the Carlitz zeta values $\zeta_A(1;s)$.
In \cite{APT}, the following generalization of Carlitz's identity (\ref{carlitzidentity}) is proved:
\begin{Theorem}\label{theorem}
For $s\geq0$, we have
$$\exp_C(\zeta_A(1;s)\omega(t_1)\cdots\omega(t_s))=P_s(t_1,\ldots,t_s)\omega(t_1)\cdots\omega(t_s),$$
where $P_s\in A[t_1,\ldots,t_s]$.
Moreover, for $s>1$, we have $P_s=0$ if and only if $s\equiv1\pmod{q-1}$.
In this case, we have
\begin{equation}\label{torsion}\zeta_A(1;s)=\frac{\widetilde{\pi}B_s}{\omega(t_1)\cdots\omega(t_s)},
\end{equation}
with $B_s\in A[t_1,\ldots,t_s]\setminus\{0\}$.
\end{Theorem}
We shall now use this result to prove:
\begin{Theorem}\label{avariant} Conjecture 
\ref{conjecturefunctional} implies the truth of the following statement:
The element $\wpi$ and all the elements $\zeta_A(1;s)$ for $s\not\equiv1\pmod{q-1}$ are 
regularly $\mu$-independent over $K_s$.
\end{Theorem}
In particular, we deduce, conditionally upon verification of Conjecture \ref{conjecturefunctional}, that the values considered are algebraically independent and also tamely $\mu$-independent
over $K_s$.
\begin{proof}[Proof of Theorem \ref{avariant}]
We choose $s_0>0$
and we denote by $\mathfrak{S}$ the set of integers $s$ with $s_0\geq s\geq 0$ such that $s\not\equiv1\pmod{q-1}$; we choose $s_0$ so that $\mathfrak{S}$ is non-empty. We also denote  by $\omega_s$, for $s\in\mathfrak{S}$, the product
$\omega(t_1)\cdots\omega(t_s)$, which belongs to $\TT_{s_0}$. 
Observe that, by the difference equation (\ref{differenceomega}), for all $s\in\mathfrak{S}$, 
 $$\mu^e(\omega_{s})=(t_1-\theta)\cdots(t_{s}-\theta)\omega_{s},$$ which implies that $\omega_{s}$
is tamely, thus regularly $\mu$-algebraic over $K_s$. 

For fixed $s\in\mathfrak{S}$, the function $\omega_s$  also defines a meromorphic function
$\CC_\infty^{s_0}\rightarrow\CC_\infty$ whose polar divisor $D_s$ is, by Proposition
\ref{omegagamma}, given by the free sum (with multiplicity one) of the affine subsets of co-dimension
one
$$D_{i,k,s}=\{(x_1,\ldots,x_{i-1},\theta^{q^k},x_{i+1},\ldots,x_s,x_{s+1},\ldots,x_{s_0}):x_j\in\CC_\infty\forall j\}\subset\CC_\infty^{s_0},$$ for $k\geq 0$ and $1\leq i\leq s$, $s\in\mathfrak{S}$. We note that if $s<s'$
and $s,s'\in\mathfrak{S}$, then $D_s\subset D_{s'}$ and $D_{s'}\setminus D_s$ has infinitely many irreducible components.

We claim that the elements $\wpi$ and $\omega_s\zeta_A(1;s)$ with $s\in\mathfrak{S}$ are $\FF_q(t_1,\ldots,t_{s_0})[\theta]$-linearly independent. We write $\mathfrak{S}=\{s_1,\ldots,s_n\}$. Let us assume by contradiction that there exists a non-trivial linear dependence relation
$$a_0\wpi+a_1\omega_{s_1}\zeta_A(1;s_1)+\cdots+a_n\omega_{s_n}\zeta_A(1;s_n)=0,\quad a_i\in A[t_1,\ldots,t_{s_0}],s_i\in\mathfrak{S}.$$
For all $\underline{k}:=(k_1,\ldots,k_{s_0})\in\NN^{s_0}$, the congruence condition on $s\in\mathfrak{S}$ implies that the evaluation $\operatorname{ev}_{\underline{k}}(\zeta_A(1;s))$ of $\zeta_A(1;s)$  at the point $(t_1,\ldots,t_s)=(\theta^{q^{k_1}},\ldots,\theta^{q^{k_s}})$ is non-zero (this follows from a well known result of Goss, in \cite{GOS}). This means that for all $j$, the function
$\omega_{s_j}\zeta_A(1,s_j)$ has polar divisor given by $D_{s_j}$. Since 
the locus of the zeros of $a_j$ has only finitely many irreducible components and since the polar divisors of the functions $\omega_{s_j}\zeta_A(1,s_j)$ are embedded one in the other along the total order of $\mathfrak{S}$ induced by $<$ on $\NN$, we deduce a contradiction. This shows that $a_0=a_1=\cdots=a_n=0$ and the 
elements $\widetilde{\pi}$ and $\omega_{s_j}\zeta_A(1,s_j)$ are linearly independent as expected. 

We set $f_j=\omega_{s_j}\zeta_A(1,s_j)$, $f_0=\widetilde{\pi}$
and $g_j=\exp_C(f_j)$ for $0\leq j\leq n$ (note that $g_0=0$).
By Theorem \ref{theorem},
 $$\exp_C[\omega_s\zeta_A(1;s)]\in \omega_sA[t_1,\ldots,t_s],\quad s\geq 0.$$
 In particular, $g_1,\ldots,g_n$ are tamely $\mu$-algebraic over $K_{s_0}$ (they satisfy linear $\mu$-diffe\-rence equations with coefficients in $K_{s_0}$).
 Conjecture \ref{conjecturefunctional} then implies that $f_1,\ldots,f_n$
 are regularly $\mu$-independent over $L_{s_0}$. Now, the theorem follows from the fact that $\omega_{s}\in L_{s_0}$ for all 
 $s\leq s_0$ and that $K_{s_0}\subset L_{s_0}$.
\end{proof}
The next Proposition goes in the direction of Theorem \ref{avariant}.
The statement is so much weaker, but holds unconditionally and is deduced from an entirely different statement: Chang and Yu's Theorem \ref{changyu}.

\begin{Proposition}\label{simpleproposition}
Assuming that $q=p$, the elements $\widetilde{\pi}$ and $\zeta_A(1;s)$ for  $0\leq s\leq p-1,s\neq 1$ are tamely $\mu$-independent over $K_s$.
\end{Proposition}

It is easy to see that, if $s=1$, $\zeta_A(1;1)$ and $\widetilde{\pi}$ are regularly $\mu$-dependent.
Indeed, by the formula 
$$\zeta_A(1;1)=\frac{\widetilde{\pi}}{(\theta-t)\omega}$$ that can be found in \cite{PEL},
we see that the regular $\mu$-polynomial 
$$P(X_1,X_2)=(\theta^q-t)\mu^e(X_2)-X_1^{q-1}X_2\in K_1[X_1,X_1]_\mu$$
satisfies $(\widetilde{\pi},\zeta(1;1))\in Z(P)$.

\begin{proof}[Proof of Proposition \ref{simpleproposition}]
Assume by contradiction that the statement is false. Then there exists
a non-trivial relation $P(f_0,f_1,f_3\ldots,f_p)=0$ with $P\in K(t_1,\ldots,t_s)[\underline{X}]_\mu$ tame and non-zero, with $f_0=\widetilde{\pi},f_1=\zeta_A(1),f_3=\zeta_A(1;2),\ldots,f_p=\zeta_A(1;p-1)$
(this time, $\underline{X}=(X_0,X_1,X_3,\ldots,X_{p-1})$ and note that $X_2$ is missing because it corresponds to $\zeta_A(1;1)$ that we have discarded). Since the subset $$\{(\theta^{p^{-k_1}},\ldots,\theta^{p^{-k_{p-1}}});\underline{k}=(k_1,\ldots,k_{p-1})\in\NN^{p-1}\}\subset\CC_\infty^{p-1}$$
is Zariski-dense, there exists a subset $\mathcal{I}$ of $\NN^{q-1}$ (of positive density)
such that the evaluation $\operatorname{ev}_{\underline{k}}(c_{\underline{i}})$ of a coefficient $c_{\underline{i}}$ of $P$ at the point $$(t_1,\ldots,t_{p-1})=(\theta^{p^{-k_1}},\ldots,\theta^{p^{-k_{p-1}}})$$
for all $\underline{k}=(k_1,\ldots,k_{p-1})\in\mathcal{I}$ is non-zero. We choose such a $p$-tuple.
Further we have, for $m\geq 0$, $n>0$, $0\leq s'\leq p-1$ and $k:=\max\{k_1,\ldots,k_{p-1}\}$, 
$$\operatorname{ev}_{\underline{k}}(\mu^m(\zeta_A(n;s')))=
\zeta_A(n-p^{-m-k_1}-\cdots-p^{-m-k_{s'}}))^{p^m}=\zeta_A(p^{m+k}n-p^{k-k_1}-\cdots-p^{k-k_{s'}})^{p^{-k}}.$$
Now, it is easy to see, due to our simple choice of parameters (especially the fact that $s'\leq p-1$ and $s'\neq1$), that our assumption yields 
a non-trivial algebraic dependence relation among $\widetilde{\pi}$ and the
values $$\zeta_A\left(p^{m+k}-\sum_{i=1}^{s'}p^{k-k_i}\right),\quad m\geq 0,\quad 0\leq s'\leq p-1,s'\neq 1,$$ where $k:=\max\{k_1,\ldots,k_{p-1}\}$. First of all, we 
can choose $\underline{k}=(k_1,\ldots,k_{p-1})$ so that, for all $i$, $k_i>0$.
This implies that for all $m$ and $s'$, the integer $$n_{m,s'}:=p^{m+k}-\sum_{i=1}^{s'}p^{k-k_i}$$ is positive. Further, for all such $m,s'$, we have that
$n_{m,s'}\not\equiv0\pmod{p-1}$; this already avoids the existence of a relation of the Euler-Carlitz-type (\ref{Carlitz}). In order to avoid
the degenerate sum shuffle relations (\ref{trivialsumshuffle}) we need to 
show that $n_{m,s'}=n_{m',s''}p^j$ for $1\leq s',s''\leq p-1$ and $j,m,m'\geq 0$
implies $m=m'$, $j=0$ and $s'=s''$. But assuming that an identity 
$$p^{m+k}-\sum_{i=1}^{s'}p^{k-k_i}=p^j\left(p^{m'+k}-\sum_{i=1}^{s''}p^{k-k_i}\right)$$
holds with $j\geq 0$ implies $j=0$. Indeed, there exists some $i$ such that $k_i=k$,
and the set of such indices $i$, non-empty, has cardinality $<p$, so that the left-hand side 
is not divisible by $p$.
Note that the condition $s',s''\leq p-1$ is crucial. The conclusion follows from Chang and Yu's Theorem \ref{changyu}.
\end{proof}

In a similar vein, we have the next Lemma.
\begin{Lemma}\label{acounterexample}
Supposing that $p=q>2$, we have that $f=\sum_{a\in A^+}a(t)^2a^{-p}\in\TT$ is 
transformally $\mu$-transcendental, hence regularly $\mu$-transcendental over $K_1$.
\end{Lemma}
\begin{proof}
We assume by contradiction that  there is a non-trivial relation 
of algebraic dependence
$$\sum_{\underline{i}\in\NN^k}c_{\underline{i}}f^{i_0}\mu(f)^{i_1}\cdots\mu^{k-1}(f)^{i_{k-1}},\quad 
c_{\underline{i}}\in K_1$$ (the sum is finite). We can suppose, without loss of generality, that 
all the coefficients $c_{\underline{i}}$ are in $A[t]$ and that they are not all divisible 
by $t-\theta$. This means that the evaluation at $t=\theta$ of the coefficients $c_{\underline{i}}$
yields a non-zero vector with entries in $A$.
 Then observing that the evaluation at $t=\theta$ of $\mu^j(f)$
is equal to the Carlitz zeta value $\zeta_A(p^{j+1}-2)$, we have thus a non-trivial 
relation of algebraic dependence with coefficients in $A$ of the Carlitz zeta values
$\zeta_A(p^j-2)$ for $j=1,\ldots,k$. Since $p-1,p\nmid p^j-2$ for all $j>0$, this 
is again in contradiction with Chang and Yu's Theorem \ref{changyu}.
\end{proof}

\subsubsection{Two more conjectures}

Additionally, we propose the following, for general $q$:

\begin{Conjecture}
Any finite subset of the set whose elements are $\widetilde{\pi},\zeta_A(n;s)$ with $p\nmid n$
and $n\not\equiv s\pmod{q-1}$ is regularly $\mu$-independent over $K_s$. 
\end{Conjecture}

This conjecture does not seem to be a consequence of Conjecture \ref{conjecturefunctional} and is
probably also quite a difficult one. In order to present something which is perhaps provable in the near future, we mention that
Conjecture \ref{conjecturefunctional} implies, for $n=1$, the following conjecture:
\begin{Conjecture}\label{hermitelindemann}
If $f\in\mathbb{K}_s\setminus\{0\}$, then either $f$ or $\exp_C(f)$ is regularly $\mu$-transcenden\-tal over $K_s$.
\end{Conjecture}

The author presently does not know if this result can be directly deduced from the analogue 
of the theorem of Hermite-Lindemann for the Carlitz exponential over $\CC_\infty$ of \cite{THI}.

\section{Transcendence degree of difference subfields}\label{transcendencedegree}

Let $\underline{f}=(f_1,\ldots,f_n)$ be an element of $\KK_s^n$. 
We suppose that for all $i$, $g_i=\exp_C(f_i)\in\KK_s$ are regularly $\mu$-algebraic. We address the following loosely question.
\begin{Question} Assume that $s>0$. Compute
the transcendence degree over $K_s$ of the $\mu$-field 
$$K_s(g_1,\ldots,g_n)_\mu$$ from the knowledge of $f_1,\ldots,f_n$.\end{Question}

Answering is likely to be difficult. The transcendence degree behaves wildly if $s>0$ and does not seem to
be in transparent (conjectural) relation with the linear forms over $\FF_q(t)[\theta]$ satisfied by the various $f$, as we are going to see. 
From now on, we are going to assume that $s=1$ for commodity. Therefore, we are
going to write $t=t_1$, $\TT=\TT_1$ and $\KK=\KK_1$.

We recall that the $\CC_\infty$-algebra $\TT$ is stable under the action of the divided derivatives
$\mathcal{D}_n$, which are, for all $n\geq 0$, the $\CC_\infty$-linear endomorphisms
uniquely defined by the rule $$\mathcal{D}_nt^m=\binom{n}{m}t^{m-n}.$$

\begin{Proposition}[Analogue of Hölder's Theorem for the function $\omega$]
We set $f_i=\frac{\widetilde{\pi}}{(\theta-t)^{i+1}}$ for $i\geq 0$. The following properties hold:
\begin{enumerate}
\item For all $i\geq 0$, $\exp_C(f_i)=\mathcal{D}_i(\omega)$.
\item The functions $\mathcal{D}_i(\omega)$ define the entries of a solution of the $\tau$-difference system
$$\tau(X_i)=(t-\theta)X_i+X_{i-1},\quad i\geq 0,\quad X_{-1}:=0.$$
\item The functions $\mathcal{D}_i(\omega)$ are algebraically independent over $K(t)$.
More precisely we have, for all $n\geq 0$,
\begin{eqnarray*}
\operatorname{transf}\deg_{K(t)}K(t)(\mathcal{D}_i(\omega):0\leq i\leq n)_\tau&=&0,\\
\operatorname{tr}\deg_{K(t)}K(t)(\mathcal{D}_i(\omega):0\leq i\leq n)_\tau&=&n+1.
\end{eqnarray*}
\end{enumerate}
\end{Proposition}

\begin{proof}
The first two properties follow easily from the fact that, over $\TT$,
$$\exp_C\frac{d}{dt}=\frac{d}{dt}\exp_C,$$ and from the difference equation
(\ref{differenceomega}). For (3), since $\omega$ is transcendental over $K(t)$
we can argue by induction and suppose that $\omega,\mathcal{D}_1(\omega),\ldots,\mathcal{D}_{n-1}(\omega)$ are algebraically independent over $K(t)$. Let us suppose by contradiction
that $\mathcal{D}_{n}(\omega)$ is algebraically dependent of the previous functions;
then there exists a non-zero irreducible polynomial $P$ of $K(t)[X_0,\ldots,X_n]$
which vanishes at the point determined by the functions $\omega,\ldots,\mathcal{D}_n(\omega)$.
The property (2) of the proposition implies that, if we denote by $P^\tau$ the polynomial 
of $K(t)[X_0,\ldots,X_n]$ obtained by letting $\tau$ act on the coefficients,
Then $P\mid P^\tau$ and in fact, $P^\tau=(t-\theta)^{d}P$ with $d$ the total degree of $P$ 
in its indeterminates, looking at the homogeneous part of highest degree of $P$. But looking at the monomials of smallest total degree of $P$, we
see that this is impossible. The property concerning the transformal dependence is easy and left to the reader.
\end{proof}
Hence, the linear relations over $\FF_q(t)[\theta]$
of $f_1,\ldots,f_n\in\KK$ do not have much to say about the transcendence degree over $K(t)$
of the $\tau$-difference field generated by $g_1=\exp_C(f_1),\ldots,g_n=\exp_C(f_n)$.

\subsection{Solving linear difference equations in $\KK$}

We are going to discuss how the Carlitz exponential $\exp_C$ can be used to solve 
the equations $$\tau^{-1}(X)=X+g,\quad g\in\KK.$$
\begin{Lemma}\label{solving}
Let $g$ be an element of $\KK$. Let $v\in\KK$ be such that
$$\exp_C(v)=-\tau(g)(t-\theta)\omega.$$
Then all the solutions $x\in\KK$ of the equation
$$\tau^{-1}(X)=X+g$$
are the elements of the set
$$\mathcal{T}(g):=\omega^{-1}\exp_C\left(\frac{v}{\theta-t}\right)+\FF_q(t).$$ 
\end{Lemma}
\begin{proof}
All we need to show is that $f_v:=\omega^{-1}\exp_C\left(\frac{v}{\theta-t}\right)$ is a solution
of our equation. But:
\begin{eqnarray*}
\tau(f_v)&=&\tau\left(\exp_C\left(\frac{v}{\theta-t}\right)\right)\tau(\omega)^{-1}\\
&=&\left(C_{\theta-t}\left(\exp_C\left(\frac{v}{\theta-t}\right)\right)-(\theta-t)\exp_C\left(\frac{v}{\theta-t}\right)\right)\tau(\omega)^{-1}\\
&=&\left(\exp_C(v)-(\theta-t)\exp_C\left(\frac{v}{\theta-t}\right)\right)((t-\theta)\omega)^{-1}\\
&=&f_v+\frac{\exp_C(v)}{(t-\theta)\omega}\\
&=&f_v-\tau(g),
\end{eqnarray*}
and the equation $\tau(X)=X-\tau(g)$ has the same set of solutions as our equation.
\end{proof}
\begin{Remark}
{\em We have that $$\mathcal{T}(g)=\omega^{-1}\exp_C\frac{1}{\theta-t}\exp_C^{-1}((\theta-t)\omega\tau(g)).$$}
\end{Remark}
We define two towers of field extensions of $K^{ac}(t)$, inductively. The first 
one is defined by $L_0=K^{ac}(t)$ and, for all $i$, 
$$L_i=L_{i-1}(\omega)(f\in\KK:\tau^{-1}(f)=f+\omega^lg:l\in\ZZ,g\in L_{i-1}).$$
For the second one, we set
again, $M_0=K^{ac}(t)$ and then inductively,
$$M_i=M_{i-1}(\cup_{l\in\ZZ}\exp_C(M_{i-1}\omega^l)\cup\cup_{l'\in\ZZ}\exp_C^{-1}(M_{i-1}\omega^{l'})).$$
Then writing $L_\infty=\cup_iL_i$ and $M_\infty=\cup_iM_i$, we obtain
$L_\infty\subset M_\infty$ as a consequence of Lemma \ref{solving}.
In the next subsection, we are going to observe that the entries of the rigid analytic trivializations
associated to {\em Chang's multiple polylogarithms}  (as in \cite{CHA1}) are all in the field $L_\infty$.

\subsection{Polylogarithm $t$-motives}

An important feature of the entire function $$\Omega:=\tau(\omega)^{-1}$$ is highlighted in 
the paper of Anderson, Brownawell and Papanikolas \cite{ABP} and
in the paper of Papanikolas \cite{PAP}: $\Omega$ is a {\em rigid
analytic trivialization} of the {\em Carlitz dual $t$-motive} $C$. The reader can consult \cite[\S 1.1.2 and \S 3.4]{PAP} for the definition and the basic properties of {\em Anderson (dual) $t$-motives} (note that the operator $\tau^{-1}$ is used therein).
The basic properties of the function $\Omega$ and of the Carlitz $t$-motive
are discussed in \cite[\S 3.1.2]{ABP} and \cite[\S 3.4.3, \S 3.3.4]{PAP},
while the notion of rigid analytic trivializations of $t$-motives
is discussed in \cite[\S 1.1.3 and \S 3.3]{PAP}.

In essence, the rigid analytic triviality of an object $M$ of the Tannakian category $\mathcal{T}$ over $\FF_q(t)$ of {\em $t$-motives} (containing the category of Anderson's $t$-motives and introduced by Papanikolas, see \cite[\S 1.1.6 and \S 3]{PAP}) allows one to realize $M$ as the solutions $\Psi\in\GL_r(\mathbb{K})$ (for some $r$) of a linear $\tau^{-1}$-difference system
 \begin{equation}\label{rigantriv}\tau^{-1}(\Psi)=\Phi\Psi,\end{equation}
 with $\Phi\in\GL_r(\CC_\infty(t))$.

Let $\Psi$ be a rigid analytic trivialization associated to a multiple polylogarithm
as in \cite{CHA1}, that is, a matrix $\Psi\in\operatorname{GL}_{d+1}(\KK)$
satisfying the $\tau^{-1}$-difference system
(\ref{rigantriv})
with $\Phi\in\operatorname{GL}_{d+1}(K^{ac}(t))$ as in \cite[(5.3.3)]{CHA1}
with $s_1,s_2,\ldots,s_d\in\ZZ_{>0}$ and $Q_1,\ldots,Q_d\in K^{ac}(t)$.
We shall prove:
\begin{Proposition}\label{inclusioninlinfty}
The entries of $\Psi$ belong to $L_\infty$.
\end{Proposition}
\begin{proof}
If we write
$$\Psi=\begin{pmatrix}\Omega^{s_1+\cdots+s_d} & 0 & \cdots & 0 & 0 \\
x_{1,0} & \Omega^{s_2+\cdots+s_d} & \cdots & 0 & 0 \\
x_{2,0} & x_{2,1} & \cdots & 0 & 0 \\
\cdots & \vdots & & \vdots  & \vdots \\
x_{d,0} & x_{d,1} & \cdots & x_{d,d-1} & 1\end{pmatrix},$$
solving the system (\ref{rigantriv}) for this choice of $\Phi$ amounts to solve
the iterated system of $\tau^{-1}$-difference equations
$$\tau^{-1}(x_{i,j})=\tau^{-1}(Q_i)(t-\theta)^{s_i+\cdots+s_d}x_{i-1,j}+(t-\theta)^{s_{i+1}+\cdots+s_d}x_{i,j},\quad 0\leq j<i\leq d,$$ setting also $x_{i,i}=\Omega^{s_{i+1}+\cdots+s_d}$ for $i=1,\ldots,d$.
This system is equivalent, by setting
$$y_{i,j}=\frac{x_{i,j}}{\Omega^{s_{i+1}+\cdots+s_d}},$$
to the system of equations:
$$\tau^{-1}(y_{i,j})=\tau^{-1}(Q_i)((t-\theta)\Omega)^{s_i}y_{i-1,j}+y_{i,j},\quad  0\leq j<i\leq d,$$
and the result follows.
\end{proof}
Before Schanuel's conjecture, there was a conjectural statement by Gelfond (see appendix of Waldschmidt's paper \cite{WAL}) which looked as a very intricate statement
involving iteration of exponentials and logarithms. Thinking about this lets us 
appreciate the simplicity and the strength of Schanuel's Conjecture.
In what concerns, for example, a control of the transcendence degree 
of the field generated by the entries of a rigid analytic trivialization of a 
multiple polylogarithm motive, we are led to understand the algebraic relations 
between elements of the field $M_\infty$. It seems thus that
we are somewhat back to Gelfond's  starting point.
\subsection*{Acknowledgements}
The author is thankful to Laurent Denis and Paul Voutier for interesting and useful discussions, and Amador Martin-Pizarro for having suggested the references \cite{MAR} and \cite{SCA} and for several remarks that have helped to improve the paper.
This research started while the author was visiting the MPIM (Bonn) in February 2016 and the IHES (Bures-sur-Yvette) in April 2016. The author is thankful to both institutions for the very good environment that positively influenced the quality of this work.

\subsection{Appendix. Some ring structures related to base-$p$ digits}\label{appendix}
We begin with a commutative, unitary ring $R$. For a symbol 
$Y$, we introduce, in the $R$-module $R[Y_0,Y_1,\ldots]$ of polynomials in infinitely 
many indeterminates $Y_0,Y_1,\ldots$ with coefficients in $R$, the monomials
$$\langle Y\rangle^i:=\prod_{j=0}^\infty Y_j^{i_j},\quad i\in\NN,\quad i=i_0+i_1p+\cdots+i_rp^r,$$ where the last expression is the expansion of $i$
in base $p$, so that $0\leq i_j\leq p-1$ for all $j$ (note that the product 
is finite for any $i$).
For $i,j\in\NN$, we set $k:=i+j$ and:
$$\langle Y\rangle^i\cdot \langle Y\rangle^j:=\langle Y\rangle^k.$$
the $R$-module 
$$R\langle Y\rangle:=\left\{\sum_{i=0}^Nc_i\langle Y\rangle^i:N\in\NN,c_i\in R\right\}$$
is now equipped with the structure of an $R$-algebra,
by using the product of monomials defined above. Explicitly, we have, for two polynomials $f=\sum_if_i\langle Y\rangle^i$ and $g=\sum_jg_j\langle Y\rangle^j$,
$$f\cdot g=\sum_k\langle Y\rangle^k\sum_{i+j=k}f_ig_j$$
(note that $R\langle Y\rangle$ is a submodule but not a subring of $R[Y_0,Y_1,\ldots]$).
This means that $R\langle Y\rangle$ is isomorphic to the polynomial
algebra $R[Z]$ in one indeterminate $Z$ with coefficients in $R$.
In particular, if $R$ is a field, $R\langle Y\rangle$ is an $R$-algebra of dimension one.

Let $R$ be an $\FF_p$-algebra, together with an 
$\FF_p$-linear injective endomorphism
$R\xrightarrow{\mu}R$. We define a $\mu$-difference structure over
$R\langle Y\rangle$ by setting 
$$\mu\left(\sum_if_i\langle Y\rangle^i\right):=\sum_i\mu(f_i)\langle Y\rangle^{pi}.$$ It is easy to see that, in this way, the isomorphism $R\langle Y\rangle\cong R[Z]$
becomes a $\mu$-difference algebra isomorphism.

We assume, from now on, that $R$ is a commutative field extension of $\FF_q$.
Let $\mathcal{P}$ be the 
ideal of $R[X]_\mu$ generated by the polynomials $X^{p^k}-\mu^k(X)$,
for all $k\geq 0$. Then $\mathcal{P}$ is a $\mu$-invariant ideal.
\begin{Lemma}
The ideal $\mathcal{P}$ is a prime ideal, and the quotient 
$R[X]_\mu/\mathcal{P}$ is isomorphic, as an $R$-algebra which also is a $\mu$-difference algebra,
to the $R$-algebra $R\langle Y\rangle$.
\end{Lemma}
\begin{proof}
All we need to show, is that there is an $R$-algebra isomorphism $R[Z]\xrightarrow{\phi}R[X]_\mu/\mathcal{P}$. Indeed, it follows in this way that $R[X]_\mu/\mathcal{P}$
is integral.

We construct $\phi$ in the following way. Let $i\in\NN$ be with base-$p$ expansion
$i=i_0+i_1p+\cdots+i_rp^r$, with $0\leq i_j<p$ for all $j$. We define the map
$\phi$ to be $R$-linear, with $$\phi(Z^i)=X^{i_0}\mu(X)^{i_1}\cdots\mu^r(X)^{i_r}\pmod{\mathcal{P}_X}.$$
We need to show that $\phi$ is multiplicative. 
Let $i,j,k\in\NN$ be such that $i+j=k$. We further expand $i,j,k$ in base $p$; let $r$ be an integer 
such that $i=\sum_{n=0}^ri_np^n,j=\sum_{n=0}^rj_np^n,k=\sum_{n=0}^rk_np^n$, with $0\leq i_n,j_n,k_n<p$.
We define the sequence $(b_n)_{n\geq 0}$ of integers to be that 
characterizing the carry over of the base-$p$ expansion of the sum $i+j=k$, namely,
$b_n\in\NN$ is defined by $b_{-1}:=0$ and, inductively for $n\geq 0$, by:
$$i_n+j_n+b_{n-1}=k_n+pb_n,\quad n\geq 0.$$
We have:
\begin{small}
\begin{eqnarray*}
\phi(Z^i)\phi(Z^j)&=&X^{i_0+j_0}\mu(X)^{i_1+j_1}\cdots \mu^r(X)^{i_r+j_r}\\
&\equiv& X^{k_0}\mu(X)^{i_1+j_1+b_0}\mu^2(X)^{i_2+j_2}\cdots \mu^r(X)^{i_r+j_r}\pmod{(X^p-\mu(X))}\\
&\equiv& X^{k_0}\mu(X)^{k_1}\mu^2(X)^{i_2+j_2+b_1}\cdots \mu^r(X)^{i_r+j_r}\pmod{(X^p-\mu(X),\mu(X)^p-\mu^2(X))}\\
&\equiv& X^{k_0}\mu(X)^{k_1}\mu^2(X)^{i_2+j_2+b_1}\cdots \mu^r(X)^{i_r+j_r}\pmod{(X^p-\mu(X),X^{p^2}-\mu^2(X))}\\
&\vdots & \\
&\equiv & X^{k_0}\mu(X)^{k_1}\cdots \mu^r(X)^{k_r}\pmod{\mathcal{P}_X}\\
&=&\phi(Z^k)
\end{eqnarray*}
\end{small}
so that $\phi$ is multiplicative, and it is obviously surjective and injective
(this is also easy to verify on the structure of the ring $R\langle Y\rangle$).
The verification
of the remaining properties is left to the reader.
\end{proof}

\end{document}